\documentclass[11pt]{amsart}

\usepackage{amsthm}
\usepackage{amssymb}
\usepackage{latexsym}
\usepackage{euscript}
\usepackage{amsmath}
\usepackage{amstext}
\usepackage{amsgen}
\usepackage{amsbsy}
\usepackage{amsopn}
\usepackage{amsfonts}
\usepackage{hyperref}
\usepackage{graphicx}
\usepackage{bbm}
\usepackage{tikz} 
\usepackage{tikz-cd}
\usepackage{float}

\usepackage{enumerate}
\usepackage{color}
\usepackage{a4wide}
\usepackage{mathrsfs}
\usepackage{multirow}
\usepackage{mathtools}

\hypersetup{
    colorlinks=true,
    linkcolor=red}

\usetikzlibrary{matrix,arrows,decorations.pathmorphing}

\newtheorem{theorem}{Theorem}[section]
\newtheorem{proposition}[theorem]{Proposition}
\newtheorem{corollary}[theorem]{Corollary}
\newtheorem{lemma}[theorem]{Lemma}

\theoremstyle{definition}
\newtheorem{definition}[theorem]{Definition}
\newtheorem{example}[theorem]{Example}

\newtheorem{remark}[theorem]{Remark}

\newcommand{\Z}{\mathbb{Z}}

\newcommand{\R}{\mathbb{R}}
\newcommand{\C}{\mathbb{C}}

\newcommand{\kk}{\mathbbm{k}}
\newcommand{\PP}{\mathbb{P}}

\newcommand{\im}{\mathrm{im}}

\newcommand{\cat}[1]{\mathrm{#1}}
\renewcommand{\mod}[1]{{#1}\mathrm{-mod}}

\newcommand{\perv}[2]{{}^{#1}\cat{Perv}\!\left(#2\right)}

\newcommand{\pfun}[2]{{}^{#1}{#2}}

\newcommand{\ic}[2]{\pfun{#1}{\sh{IC}_{#2}}}

\newcommand{\constr}[2]{\cat{D}_{#1}({#2})}
\newcommand{\der}[2]{\cat{D}^{#1}({#2})}

\newcommand{\Mor}[3]{{\mathrm{Hom}}_{#1}\!\left(#2,#3\right)}

\newcommand{\Ext}[4]{{\mathrm{Ext}_{#1}^{#2}}\!\left(#3,#4\right)}
\newcommand{\End}[1]{\mathrm{End}\!\left(#1\right)}

\newcommand{\sh}[1]{\mathcal{#1}}

\newcommand{\pt}{{\bf pt}}

\newcommand{\stan}[2]{\pfun{#1}{\Delta_{#2}}}
\newcommand{\costan}[2]{\pfun{#1}{\nabla_{\! #2}}}

\newcommand{\stab}[1]{\mathrm{Stab}(#1)}

\title[Wall and chamber structure for perverse sheaves on projective spaces]{Wall and chamber structure for a special biserial algebra coming from perverse sheaves on $\mathbb{P}^n$}
\author{Alessio Cipriani, Martina Lanini}
\address{Dipartimento di Informatica - Settore di Matematica, Universit\`a degli Studi di Verona, Strada le Grazie 15 - Ca’ Vignal, I-37134 Verona, Italy}
\email{alessio.cipriani@univr.it}
\address{Dipartimento di Matematica, Universit\`a di Roma “Tor Vergata”, Via della Ricerca Scientifica 1, I-00133 Rome, Italy}
\email{lanini@mat.uniroma2.it}

\begin{document}

\maketitle

\begin{abstract}
    We describe the wall and chamber structure of a special biserial algebra whose module category is equivalent to the category of (middle) perverse sheaves on the complex projective space $\mathbb{P}^n$.
    In particular, by the well known classification of indecomposable modules for special biserial algebras, we deduce that the algebra of interest is of finite representation type and we provide an explicit description of the walls of the structure.
    By a result of Bridgeland this wall and chamber structure coincides with the chamber structure in an open subset of the space of stability conditions on the bounded derived category of constructible sheaves on $\mathbb{P}^n$.
\end{abstract}

\section{Introduction}
This work stems from the desire of understanding Bridgeland space of stability conditions for the bounded derived category of complexes on complex projective spaces constructible with respect to the Schubert stratification. Stability spaces for triangulated categories were introduced by Bridgeland in \cite{bridgeland07} inspired by the work of Douglas on string theory \cite{douglas}. Many progresses have been made in the study of stability spaces for derived categories of coherent sheaves, while the case of bounded derived categories of constructible complexes, treated in the $\mathbb{P}^1$ case in \cite{MR2739061}, has had less fortune so far. One of the reasons to look at the bounded derived category $\constr{c}{X}$ of constructible sheaves on some variety $X$ is that, if the stratification considered on $X$ is by affine subvarieties whose cohomology is concentrated in one degree (e.g. flag varieties with Schubert stratification), then the abelian category of middle perverse sheaves on $X$ is a faithful heart of $\constr{c}{X}$, that is a heart whose derived category is $\constr{c}{X}$  itself.  By \cite{MR4349392} the category of middle perverse sheaves can be realised as a module category for a quiver with relations, so that our triangulated category of interest is now the derived category of a quiver module category, in particular of modules for a finite dimensional algebra.

Stability conditions for the category of modules over a finite dimensional algebra were defined and investigated since the Nineties \cite{king}. The definition of wall and chamber structure for a finite dimensional algebra is due to Bridgeland \cite{bridgelandWC} and links the geometric and algebraic notions of stability condition spaces in the following sense: understanding the wall and chamber structure for the category $(Q,I)-\textrm{rep}$ of representations of a quiver with relations is equivalent to understanding wall crossing phenomena in an open subset of Bridgeland stability space for the derived category of $(Q,I)-\textrm{rep}$. In fact in some cases it is enough to understand the wall and chamber structure to obtain all the needed information to control Bridgeland's space of stability conditions \cite{bousseau}.

The abelian category of middle perverse sheaves on the complex projective space $\mathbb{P}^n$ with Schubert stratification is known to be equivalent to the category of modules for a particularly nice quiver with relations $(Q(n), I(n))$ (see for instance \cite[Example 1.1]{stroppelTQFT}) whose corresponding algebra $B(n)$ is special biserial. Special biserial algebras have been widely studied for very many years, starting from \cite{SW}, where they were defined. This is certainly due to the fact that their representation theory is extremely well understood and the classification of their indecomposable modules is purely combinatorial, as explained in \cite{Butler-Ringel, WW}. In particular, their indecomposable modules are subdivided into three classes: string modules, band modules and non-uniserial injective-projective modules. We show that our special biserial algebra $B(n)$ does not have any band module (Remark \ref{proposition : no bands}). Since the other two classes of modules are finite, we can deduce that the algebra of interest is of finite representation type. After this paper was written, we discovered that the classification of indecomposable $B(n)$-modules was exploited in \cite{Pe} to investigate exceptional $B(n)$-sequences. Nevertheless, we were not able to find a proof in the literature. As for the wall and chamber structure, the stability spaces corresponding to non thin indecomposable $B(n)$-modules are contained in the walls corresponding to thin modules, so that we are left with a wall and chamber structure which only involves indecomposable thin modules (Proposition \ref{theorem : rooty_walls}). Walls for string modules have been explicitly described in the acyclic case by \cite{asai} via facial description and in \cite{string_and_band} via the dual description by ray generators any polyhedral cones. It would have been possible to invoke \cite{string_and_band} to prove Proposition \ref{theorem : rooty_walls} and to deduce the wall and chamber structure of the category of middle perverse sheaves on the complex projective space from \cite{string_and_band} once the indecomposable thin $B(n)$-modules had been interpreted as thin modules for an acyclic quiver with relations. Nevertheless, since the combinatorics of indecomposable modules for our algebra is very explicit, we tried to keep this paper as self contained as possible (apart from the classification of indecomposable modules) and hence we provide independent proofs of the non relevance of non thin modules, as well as
the wall facial description, 
both relying on the nice combinatorics of strings on $(Q(n), I(n))$ (see Section \S\ref{sec : string combinatorics}). Such nice combinatorics translates into perverse sheaves language: the perverse sheaves corresponding to the walls of the structure can be constructed as iterated extensions of appropriate simples, standard and costandard objects (Proposition \ref{prop : perverse walls}).

Since our motivation came from geometry, we did not investigate the consequences of our main theorem to silting theory. By \cite{asai} and \cite{BST19} the chambers in the wall and chamber structure of a finite dimensional algebra are in bijection with its 2-silting complexes (up to isomorphism), and two 2-silting complexes corresponding to adjacent chambers are related by a mutation. Therefore, determining
the number of isomorphism classes of 2-silting complexes for $B(n)$ becomes the purely combinatorial problem of counting the maximal cones of the polyhedral fan complex resulting from Theorem \ref{theorem:explicit_descr_stab_spaces}. 

\subsection{Structure of the paper} In Section \S\ref{section:background} we collect some basic facts on quivers with relations and walks on them; we also define our main character $(Q(n), I(n))$. In the following Section \S\ref{sec : string combinatorics} we investigate the combinatorial behaviour of strings on $(Q(n), I(n))$ and prove that there are no bands  on $(Q(n), I(n))$. Section \S\ref{sec : RepTh of B(r)} deals with the representation theory of the finite dimensional algebra $B(n)$ corresponding to $(Q(n), I(n))$. In particular, we recall the parameterization of isomorphism classes of indecomposable modules and we apply the string combinatorics to deduce that $B(n)$ is of finite representation type with a very explicit realisation of the indecomposable objects.  In Section \S\ref{sec:WCstructure} we recall King's definition of semistable objects for a finite dimensional algebra and we describe the stability spaces (via inequalities) of any indecomposable module. The facial description of the wall and chamber structure of $B(n)$ is given in Theorem \ref{theorem:explicit_descr_stab_spaces}. Finally, Section \S\ref{sec:perverse} is about the geometric counterpart of our result: more precisely, we explicitly describe the perverse sheaves corresponding to the walls of the structure.

\subsection*{Acknowledgements} 
We thank Jon Woolf for many motivating and enlightening discussions, and Catharina Stroppel for very useful correspondence. We also thank Lidia Angeleri H\"ugel for her feedback on a first version of this paper, and Walter Mazorchuk for pointing out \cite{Pe} to us. Finally, many thanks to the anonymous referee for several very helpful comments. The work of A.C. was funded by MUR PNRR--Seal of Excellence,
CUP B37G22000800006, and the project REDCOM: \emph{Reducing complexity in algebra, logic, combinatorics}, financed by the programma Ricerca Scientifica di Eccellenza 2018 of the Fondazione Cariverona. M.L. was partially funded by the Fondi di Ricerca Scientifica di Ateneo 2021 CUP E853C22001680005, and the MUR Excellence Department Project awarded to the Department of Mathematics, University of Rome Tor Vergata, CUP E83C23000330006. Both the authors are members of the network INdAM-G.N.S.A.G.A.

\section{Reminder on special biserial algebras and walks on their quivers}\label{section:background}
We recall here some basics on special biserial algebras and their representation theory, see for example \cite{SW,WW, Butler-Ringel}.

Throughout the paper we will denote the interval of natural numbers $\{i,i+1,\ldots, j\}$ by  $[i,j]$, and use the short hand notation $[m]=[1,m]$.

A quiver $Q=(Q_0, Q_1)$ is an oriented graph, with vertex set $Q_0$ and arrow set $Q_1$. For any algebraically closed field $\kk$, we denote by $\kk Q$ its path algebra, that is the algebra obtained by concatenating paths on the graph, see for instance \cite[Chapter 2, \S1]{MR2197389}. If $F$ is the ideal of $\kk Q$ generated by the arrows, we say that a further ideal $I$ is \emph{admissible} whenever there is an $m$ such that $F^m\subseteq I\subseteq F^2$. The algebra $\kk Q/I$ is said \emph{special biserial} if $I$ is admissible and the following two extra conditions hold
\begin{itemize}
    \item every vertex is source of at most two arrows and is target of at most two arrows,
    \item for any arrow $\gamma$ there is at most one arrow $\delta$ such that the composition $\gamma\circ \delta\not\in I$ and at most one arrow $\delta'$ such that $\delta'\circ \gamma\not\in I$.  
\end{itemize}
If $\kk Q/I$ is special biserial, then $I$ is generated by monomials (that is paths on the quiver) and binomials (that is relations such as $a\gamma-b\delta$, with $a,b\in\kk^\times$ and $\gamma,\delta$ paths on $Q$), as explained in \cite[Corollary of Lemma 1]{SW}.

In this paper we deal with the quiver $Q(n)$, that is the quiver with 
\begin{itemize}
    \item vertices $i\in[0,n]$, 
        \item arrows $\alpha_i\colon  i+1\to i$ and $\beta_i\colon i\to i+1$ with $i\in[0,n-1]$,
\end{itemize}
and we consider the ideal of relations $I(n)$ generated by 
\[\beta_i\alpha_i-\alpha_{i+1}\beta_{i+1}, \quad \alpha_i\alpha_{i+1},\quad  \beta_{i+1}\beta_i\qquad (i\in[0,n-2]), \quad \alpha_0 \beta_0.
\]

\begin{remark}
  The quiver with relation $(Q(n), I(n))$ coincides with the one appearing in \cite[Example 1.1.]{stroppelTQFT}. 
\end{remark}

For instance, when $n=2$ we have the quiver
\begin{equation}\label{equation:quiver_Q}
 Q(2)=\begin{tikzcd}
    2 \ar[r,bend left=20,"\alpha_1"]& 1\ar[r,bend left=20,"\alpha_0"]\ar[l,bend left=20,"\beta_1"] & 0 \ar[l,bend left=20,"\beta_0"]
\end{tikzcd}
\end{equation}
with relations
\begin{equation}\label{equation:relations_I}
    I(2)=\langle \beta_0\alpha_0-\alpha_1 \beta_1, \alpha_0\alpha_1,\beta_1\beta_0, \alpha_0\beta_0 \rangle.
\end{equation}

We denote by $B(n)$ the quotient algebra $\kk Q(n)/I(n)$. It is immediate to see that $B(n)$ is special biserial.

Given a quiver $Q=(Q_0,Q_1)$, we denote by $Q_1^*$ the set of formal inverses of the arrows of $Q$. If $\delta\in Q_1^*$, we write $\delta^*$ to denote the arrow in $Q_1$  whose formal inverse is $\delta$ (that is $(\gamma^*)^*=\gamma$). By convention, if $\gamma\in Q_1$ has source $s(\gamma)$ and target $t(\gamma)$ then we set $s(\gamma^*)=t(\gamma)$ and $t(\gamma^*)=s(\gamma)$.

\begin{definition} Let $Q=(Q_0, Q_1)$ be a quiver.
    \begin{enumerate}
        \item A \emph{walk of length $m\geq 1$ on $Q$} is a  sequence $w=\gamma_m\ldots\gamma_2\gamma_1$ with $\gamma_k\in Q_1\cup Q_1^*$ for any $k$ and $s(\gamma_{k+1})=t(\gamma_k)$ for all $k\leq m-1$.
        \item A \emph{walk on $Q$ of length 0 on the vertex $i$} is the trivial path which stays on such a vertex. We denote it by $\epsilon_i$.
        \item A \emph{path on $Q$} is a walk $w=\gamma_m\ldots\gamma_2\gamma_1$ with all $\gamma_k\in Q_1$.
        \item A walk is said to be \emph{reduced} if either $w$ is of length 0, or $w=c_1\ldots c_m$ and $c_i\neq c_{i+1}^*$ for any $i\in [m-1]$. 
    \end{enumerate}
\end{definition}

\begin{example}\label{exe:walk&path}
    Let us consider $Q(2)$ as defined in (\ref{equation:quiver_Q}).
    \begin{enumerate}[i)]
        \item The sequence $w=\alpha_1^*\beta_0$ is a  walk on $Q(2)$ (and it is not a path).
   \item    The sequence $w=\beta_0^*\alpha_1\beta_1$ is a reduced walk on $Q(2)$ (and it is not a path).
\item The sequence $w=\beta_1\beta_0$ is a path on $Q(2)$.
    \end{enumerate}
\end{example}

If $w=c_m\ldots c_1$ is a walk of length $m>0$, we denote by $w^*$ the walk $c_1^*\ldots c_m^*$. By convention, $\epsilon_i^*=\epsilon_i$ for any $i\in Q_0$. 
\begin{definition}Let $Q=(Q_0,Q_1)$ be a quiver.
 A walk $w$ on $Q$ is said to \emph{contain a non trivial path} $p$ if $p$ or $p^*$ are contained in $w$ (that is $p$ or $p^*$ can be obtained from the word of $w$ by removing a suffix and/or a prefix).
\end{definition}
\begin{example}Let us consider $Q(2)$ as defined in (\ref{equation:quiver_Q}).
    \begin{enumerate}[i)]
        \item The walk $w=\alpha_1^*\beta_0$ contains the length 1 paths $u_1=\alpha_1$ and $u_2=\beta_0$.
        \item The walk $w=\beta_0^*\alpha_1\beta_1$ contains the path $\alpha_1\beta_1$.
        \end{enumerate}
\end{example}

\begin{definition}
    Let $Q=(Q_0,Q_1)$ and let $I$ be an admissible ideal of $\kk Q$ such that $\kk Q/I$ is special biserial.
    A \emph{string on $(Q,I)$} is a reduced walk on $Q$ such that if $p$ is a non trivial path contained in $w$ then $p$ is not a monomial appearing as a summand of any of the generators of $I$.
\end{definition}

\begin{example}Let us consider $Q(2)$ as defined in (\ref{equation:quiver_Q}).
    \begin{enumerate}[i)]
        \item The walk $w=\alpha_1^*\beta_0$ is a string as none of its subpaths appear in $I(2)$.
        \item The walk $w=\beta_0^*\alpha_1\beta_1$ is not a string as the path $\alpha_1\beta_1$ is a monomial in the binomial relation $\beta_0\alpha_0-\alpha_1\beta_1\in I(2)$.
        \item The walk $\beta_1\beta_0\in I(2)$ and hence it is not a string.
   \end{enumerate}
\end{example}

\begin{definition} Let $Q=(Q_0,Q_1)$ be a quiver.
\begin{enumerate}
\item The \emph{source} $s(w)$, resp. the \emph{target} $t(w)$, of a walk $w=\gamma_m\ldots \gamma_1$ of length $s>0$ on $Q$ is $s(\gamma_1)$, resp. $t(\gamma_m)$. The trivial walk $\epsilon_i$ has source and target equal to $i$.
\item A \emph{cycle} is a walk $w$ of strictly positive length such that $s(w)=t(w)$ 
\item A \emph{reduced cycle} is a cycle $w=\gamma_m\ldots\gamma_1$ with $\gamma_{k}\neq \gamma_{k+1}^*$ for any $k\in[m-1]$ and the indices are taken modulo $m$.
    \item If $w$ is a reduced cycle we write $w^h$ for the path obtained by concatenating $w$ with itself $h$-times.
\end{enumerate}
\end{definition}

\begin{example}
 Let $Q$ be the equioriented Dynkin quiver of type $\Tilde{A_3}$, that is the quiver given by $Q_0=\{0,1,2\}$ and $Q_1=\{\lambda_0:0\to 1, \lambda_1:1\to 2, \lambda_2:2\to 0\}$. Let $I=\langle(\lambda_2\lambda_1\lambda_0)^3\rangle$. We consider walks on $(Q,I)$ 
 \begin{enumerate}\label{exple : equioriented cycle}
     \item The source of $w=\lambda_2\lambda_1$ is $s(w)=1$, while its target is $t(w)=0$. Therefore $w$ is not a cycle.
     \item Let $w=\lambda_2\lambda_1\lambda_0$. Then $s(w)=t(w)=0$ and so $w$ is a (reduced) cycle on $Q$.
 \end{enumerate}
\end{example}

\begin{definition}
    Let $(Q=(Q_0,Q_1)$ and let $I$ be an admissible ideal of $\kk Q$ such that $\kk Q/I$ is special biserial.
    A \emph{band on $(Q,I)$} is a non trivial reduced cycle  $w$ on $Q$ such that $w^h$ is a string path  on $(Q,I)$ for any $h\geq 1$ and such that it does not exist a string path $u$ of length strictly less than the length of $w$ with $u^k=w$ for some $k> 1$.
\end{definition}

\begin{example}
Let $Q$ be the Dynkin quiver of type $\Tilde{A_3}$ given by $Q_0=\{0,1,2\}$ and $Q_1=\{\lambda_0:0\to 1, \lambda_1:1\to 2, \lambda_2:0\to 2\}$ and $I=\emptyset$. Then $\lambda_2^*\lambda_1\lambda_0$ is a band on $(Q,I)$.
\end{example}

\section{Combinatorics of strings on $(Q(n), I(n))$}\label{sec : string combinatorics}
As we will see in the next section, the representation theory of a special biserial algebra is controlled by string and band combinatorics on the corresponding quiver with relations. We collect in this section some results of strings on our quiver $(Q(n),I(n))$ and leave the checks to the reader.

\begin{lemma}
Let $w=\gamma_m\ldots \gamma_1$ be a string on $(Q(n), I(n))$ of length $m\geq 2$. Then either $w\in\{\beta_{n-1}\alpha_{n-1}, \alpha_{n-1}^*\beta_{n-1}^*\}$ or $w$ does not contain any path of length $\geq 2$.
\end{lemma}
    
    
 

\begin{corollary}\label{cor : types of strings}
There are three types of strings on  $(Q(n), I(n))$: 
\begin{enumerate}
    \item the trivial paths $\epsilon_i$, for $i\in[0,n]$,
    \item the cycles $\beta_{n-1}\alpha_{n-1}$ and $\alpha_{n-1}^*\beta_{n-1}^*$,
    \item the alternating walks, that is walks such as $\gamma_m\ldots\gamma_1$ and $\gamma_k\in Q(n)_1$ if and only if $\gamma_{k+1}\in Q(n)_1^*$ for any $k<m$.
\end{enumerate}
\end{corollary}

\begin{lemma}\label{lemma : string length 2 alternating}
    Let $w=\gamma_m\ldots\gamma_1$ be an alternating string of length $>0$ on $(Q(n), I(n))$. Let $f_w:[0,m]\to [0,n]$ be the function defined as
    \[ 
    f_w(0)=s(\gamma_1), \quad f_w(j)=t(\gamma_{j}), \ j\in [m].
    \]
    Then $f_w$ is strictly monotone.  Moreover $f_w$ is increasing if and only if $f_{w^*}$ is decreasing.
\end{lemma}


    

\begin{example}
    Let us consider the string $w=\beta_0^*\alpha_1$ on $(Q(2),I(2))$. Then, we have \[
    \begin{cases}
        f_w(0)&=s(\alpha_1)= f_{w^*}(2)=2 \\ f_w(1)&=t(\alpha_1)=f_{w^*}(1) =1 \\ f_w(2)&=t(\beta_0^*)=f_{w^*}(0)=0
    \end{cases}.
    \]
    Note that $f_w$ is strictly decreasing and that $f_{w^*}$ is strictly increasing.
\end{example}

\begin{lemma}\label{lemma : fw bijection on image}
    Let $w=\gamma_m\ldots\gamma_1$ be an alternating string of length $>0$ on $(Q(n), I(n))$ and let $f_w$ be defined as in Lemma \ref{lemma : string length 2 alternating}. If $f_w$ is increasing, then $\im f_w=[s(w), t(w)]$ and hence $f_w:[0,m]\to [s(w), t(w)]$ is bijective.
\end{lemma}

We have already observed that that if $w$ is a, then $w^*$ is also a string. The \emph{$*$-class of a string} is the set $\{w, w^*\}$.
The following lemma parameterizes the $*$-classes of strings on $(Q(n), I(n))$.
\begin{lemma}\label{ lemma : parameterisation of strings}
The set of $*$-classes of alternating walk strings on $(Q(n), I(n))$ is in bijection with the following set
    \[
    \mathscr{S}=\left\{(a,b,\eta)\mid 0\leq a<b\leq n, \ \eta\in \{\pm1\} \right\}.
    \]
\end{lemma}
\begin{proof}
    We denote by $\mathscr{X}$ the set of $*$-classes of strings on $(Q(n), I(n))$ of length $>0$. 
    We will define a bijective function  $\varphi:\mathscr{X}\rightarrow \mathscr{S}$.

    Let $w=\gamma_m\ldots \gamma_1$ be a string of length $>0$. Then $f_w$ is strictly monotone and it is increasing if and only if $f_{w^*}$ is decreasing. Since we only care about $*$-classes we can assume that $f_w$ is strictly increasing. We define 
    \[
    \varphi(w)=(f_w(0), f_w(m), \eta), \hbox{ with }\eta=\begin{cases}
        1&\hbox{ if }\gamma_1\in Q(n)_1,\\
        -1&\hbox{ if }\gamma_1\in Q(n)_1^*.
    \end{cases}
    \]
    Viceversa, if $(a,b,\eta)\in\mathscr{S}$ we set
    \[
 \psi((a,b,\eta))=\begin{cases}
     \alpha_{b-1}^*\beta_{b-2}\ldots\alpha_{a+1}^*\beta_a &\hbox{ if }\eta=1, \hbox{ and }a\equiv_2 b\\
     \beta_{b-1}\alpha_{b-2}^*\ldots\alpha_{a+1}^*\beta_a &\hbox{ if }\eta=1, \hbox{ and }a\not\equiv_2 b\\
     \beta_{b-1}\alpha_{b-2}^*\ldots\beta_{a+1}\alpha_a^* &\hbox{ if }\eta=-1, \hbox{ and }a\equiv_2 b\\
     \alpha_{b-1}^*\beta_{b-2}\ldots\beta_{a+1}\alpha_{a}^* &\hbox{ if }\eta=-1, \hbox{ and }a\not\equiv_2 b,
 \end{cases}
    \]
    where $a\equiv_2 b$, resp. $a\not\equiv_2 b$, indicates that $a$ and $b$ have, resp. do not have, the same parity.
    
  Since $t(\beta_i)=i+1=t(\alpha_{i+1})=s(\alpha_{i+1}^*)$ and $t(\alpha_j^*)=s(\alpha_j)=s(\beta_{j+1})$ then $(a,b,\eta)$ is certainly a reduced walk. It is a string since it does not contain any path of length $\geq 2$ or any monomial in $I(n)$ or summand of a binomial generating $I(n)$. Note that $s(\psi((a,b,\eta)))=a$ and $t(\psi((a,b,\eta)))=b$. At this point is an easy check to see that $\varphi$ and $\psi$ are mutual inverses.
\end{proof}

\begin{corollary}\label{cor : number of *classes of strings}
    The number of $*$-classes of strings on $(Q(n), I(n))$ is $(n+1)^2+1$.
\end{corollary}
\begin{proof}
Corollary \ref{cor : types of strings} tells us that we have three types of strings. As for the first type of strings, the trivial paths, they are fixed by the $*$-operation and there are $n$ many of them. We have only two strings of the second type which are interchanged by the $*$-operation and this gives us another $*$-class.
    
By the previous lemma, we are hence left to compute the cardinality of $\mathscr{S}$ which coincides with twice the cardinality of the set of connected subintervals of $[0,n]$, that is $n(n+1)$.
\end{proof}

\begin{remark} \label{proposition : no bands}
It follows from Corollary \ref{cor : number of *classes of strings} that there are no bands $(Q(n), I(n))$ for any $n\geq 1$, and the
algebra $B(n)$ is representation-finite. 
\end{remark}



    Since we require that $w^h$ is a string for any $h\geq 1$, in particular $w$ itself is a string and by Lemma \ref{lemma : string length 2 alternating} $f_w$ is strictly monotone. On the other hand we require that $f_w(m)=t(\gamma_m)=s(\gamma_1)=f_w(0)$, but the strictly monotonicity of $f_w$ tells us that this is not possible.

\section{Classification of indecomposable $B(n)$-modules}\label{sec : RepTh of B(r)}
We devote this section to the classification of $B(n)$-indecomposable modules. 

We start be recalling what a representation is in this setting.

Let $Q=(Q_0,Q_1)$ be a quiver and $I$ be an admissible ideal of its path algebra $\kk Q$. A \emph{$(Q,I)$-representation over $\kk$} is given by
\begin{itemize}
    \item a collection of finite dimensional $\kk$-vector spaces $(M_i)_{i\in Q_0}$
    \item a collection of $\kk$-linear maps $(\varphi_{\gamma}: M_i\to M_j)_{\gamma:i\to j\in Q_1}$
  such that the maps verify the relations in $I$: if $\sum_i c_i(\gamma_{m_i}^{(i)}\ldots \gamma_1^{(i)})\in I$ then $\sum_i c_i({M_{\gamma_{m_i}^{(i)}}\circ \ldots \circ M_{ \gamma_1^{(i)}}})=0$
\end{itemize}
If $M=((M_i)_{i\in Q_0}, (\varphi_\gamma)_{\gamma \in Q_1})$ and $N=((N_i)_{i\in Q_0}, (\varphi'_\gamma)_{\gamma \in Q_1})$ are two $(Q,I)$-representations, then a morphism from $M$ to $N$ is a collection of $\kk$-linear maps between the vector spaces of the representations $\psi=(\psi_i:M_i\rightarrow N_i)_{i\in Q_0}$ such that for any $\gamma:i\to j\in Q_1$ it holds that $\psi_j\circ\varphi_\gamma=\varphi_\gamma'\circ\psi_i$. The category of (finite dimensional) $(Q,I)$-representations is equivalent to the category of (finite dimensional) $\kk Q/I$-modules (see, for instance, \cite[Chapter II, Theorem 3.7]{MR2197389}), therefore we will very often consider $(Q(n),I(n))$-representations and refer to them as $B(n)$-modules.

Let $Q$ be a quiver and let $I$ be an admissible ideal such that $\kk Q/I$ is special biserial. As already recalled, $I$ is generated by monomials and binomials in paths on $Q$. Assume that we have a minimal set of generators of $I$. If a path $p$ appears as a monomial (rescaled by some non zero scalar) in some binomial relation in $I$ and $p\not\in I$, then this binomial relation is unique. The following theorem provides a parameterization of $\kk Q/I$-modules for any special biserial algebra:
\begin{theorem}\label{thm : classification first version}\cite[Proposition 2.3]{WW} Let $\kk Q/I$ be a special biserial algebra.  
The indecomposable $\kk Q/I$-modules are organized into three families: one parameterized by strings on $(Q,I)$, one parameterized by bands on $(Q,I)$, and a third one parameterized by (non redundant) binomial relations in $I$. Moreover, the three families have empty intersection. 
\end{theorem}

By Remark \ref{proposition : no bands} we know that there are no bands on $(Q(n), I(n))$ and since we are interested in classifying indecomposables $B(n)$-modules we can restrict our attention only to the class of modules corresponding to strings and to binomial relations. 

Let $w=\gamma_m\ldots \gamma_1$ be a string of strictly positive length on a quiver with relations $(Q,I)$ for a special biserial algebra. We can define a function $f_w:[0,m]\rightarrow [0,n]$ in the same way as in Lemma \ref{lemma : string length 2 alternating}. If $w=\epsilon_i$ for some $i$ we define $f_w:\{0\}\to [0,m]$ as $f_w(0)=i$.

Let $\gamma\in Q_1\cup Q_1^*$. We denote by $\Tilde{\gamma}\in Q_1$ the unique element in $\{\gamma,\gamma^*\}\cap Q_1$. We will refer to such $\Tilde{\gamma}$ as an  honest arrow, since it is an element of $Q_1$.

\begin{definition}\cite[\S2]{WW}  Let $w=\gamma_m\ldots\gamma_1$ be a string on $(Q(n), I(n))$. The \emph{string representation $M(w)$ of $B(n)$} is defined as follows: 

       \[
     M(w)_j=\bigoplus_{h\in f_w^{-1}(j)} \kk
    \qquad  j\in [0,n], 
    \]
    \[
    M(w)_\delta=      
    \bigoplus_{i:\gamma_i\in\{\delta, \delta^*\}}(\kk(s(\Tilde{\gamma_i}))\stackrel{\bf 1}{\rightarrow} \kk(t(\Tilde{\gamma_i}))
    \qquad \delta\in Q(n)_1 ,
    \]
    where in the definition of the maps we denote by $\kk(h)$ the $h$-th copy of $\kk$ in the direct sum $\bigoplus_{h\in f_w^{-1}(l)} \kk$.
\end{definition}

Note that the above defined representation is indecomposable by construction. Moreover, $M(w)$ is isomorphic to $M(v)$ if and only if $w=v^*$. Thus isomorphism classes of string modules are parameterized by $*$-classes of strings on $(Q(n), I(n))$.

Since $M(w)$ is simple if and only if $w=\epsilon_i$ for some $i\in [0,n]$, we will sometimes use the notation $S_i$ for $M(\epsilon_i)$. If $w$ is a string on $(Q(n), I(n))$ of length $\geq 1$ different from $\beta_{n-1}\alpha_{n-1}, \alpha_{n-1}^* \beta_{n-1}^*$ and $(a,b,\eta)$ is the element of $\mathscr{S}$ corresponding to the $*$-class of $w$, we will often write $M(a,b,\eta)$ instead of $M(w)$.

Next we need to give the definition of the indecomposable modules corresponding to binomial relations in $I(n)$. Recall that the (non redundant) binomial relations in $I(n)$ are \[ \beta_i\alpha_i-\alpha_{i+1}\beta_{i+1}, \qquad i\in [0,n-2]
\]
and are therefore in bijection with the set $[0,n-2]$.
\begin{definition}Let $n\geq 2$ and let $i\in[0,n-2]$. The \emph{(non uniserial) projective-injective $B(n)$-module $R(i)$} are:
    \[
    R(i)_j=\begin{cases}
        \kk&\hbox{ if }j\in\{i,i+2\},\\
        \kk\oplus\kk&\hbox{ if }j=i+1,\\
        (0)&\hbox{ otherwise}, 
    \end{cases}\qquad j\in[0,n],
    \]
     \[
    R(i)_\gamma=\begin{cases}
        \left[
        \begin{array}{cc}
             1 &0
        \end{array}
        \right]&\hbox{ if }\gamma\in\{\alpha_i,\beta_{i+1}\},\\
        \left[
        \begin{array}{c}
             0 \\
             1 
        \end{array}
        \right]&\hbox{ if }\gamma\in\{\beta_{i}, \alpha_{i+1}\},\\
        (0)&\hbox{ otherwise}. 
    \end{cases}\qquad \gamma\in Q(n)_1,
    \]
\end{definition}

Recall that by Remark \ref{proposition : no bands} we know that there are no band modules. A more precise version of Theorem \ref{thm : classification first version} for our algebra $B(n)$ is hence the following:
\begin{theorem}\label{thm:classification of indcps} Let $M$ be an indecomposable $B(n)$-module. Then either there exists an $i\in[0,n-2]$ such that $M\simeq R(i)$ or there exists a string $w$ on $(Q(n), I(n))$ such that $M\simeq  M(w)$.
\end{theorem}

\begin{example}\label{exple:indcps for r=1}
    In the case $n=1$, the indecomposable $B(1)$-modules are all string modules:

\begin{equation*}
\begin{tikzpicture}
\node (C) at (-1.6,0) {$S_0=$};
\node (A) at (-1,0) {$0$};
\node (B) at (0.4,0) {$\C$,};
\path[->,font=\scriptsize,>=angle 90]
(A) edge [bend left] node[above] {0} (B);
\path[->,font=\scriptsize,>=angle 90]
(B) edge [bend left] node[below] {0} (A);
\end{tikzpicture}
\qquad
\begin{tikzpicture}
\node (C) at (-1.6,0) {$S_1=$};
\node (A) at (-1,0) {$\C$};
\node (B) at (0.4,0) {$0$,};
\path[->,font=\scriptsize,>=angle 90]
(A) edge [bend left] node[above] {0} (B);
\path[->,font=\scriptsize,>=angle 90]
(B) edge [bend left] node[below] {0} (A);
\end{tikzpicture}
\end{equation*}
\begin{equation*}
\begin{tikzpicture}
\node (C) at (-2,0) {$M(\alpha_0)=$};
\node (A) at (-1,0) {$\C$};
\node (B) at (0.4,0) {$\C$,};
\path[->,font=\scriptsize,>=angle 90]
(A) edge [bend left] node[above] {1} (B);
\path[->,font=\scriptsize,>=angle 90]
(B) edge [bend left] node[below] {0} (A);
\end{tikzpicture}
\quad
\begin{tikzpicture}
\node (C) at (-2,0) {$M(\beta_0)=$};
\node (A) at (-1,0) {$\C$};
\node (B) at (0.4,0) {$\C$,};
\path[->,font=\scriptsize,>=angle 90]
(A) edge [bend left] node[above] {0} (B);
\path[->,font=\scriptsize,>=angle 90]
(B) edge [bend left] node[below] {1} (A);
\end{tikzpicture}
\quad
\begin{tikzpicture}
\node (C) at (-2.3,0) {$M(\beta_0\alpha_0)=$};
\node (A) at (-1,0) {$\C^2$};
\node (B) at (0.4,0) {$\C$,};
\path[->,font=\scriptsize,>=angle 90]
(A) edge [bend left] node[above] {B} (B);
\path[->,font=\scriptsize,>=angle 90]
(B) edge [bend left] node[below] {A} (A);
\end{tikzpicture}
\end{equation*}
where $\mathrm{A}=\left[\begin{array}{c}
          1\\ 0\end{array}\right]$ and $\mathrm{B}=\left[\begin{array}{cc}
          0 & 1\end{array}\right]$. 
\end{example}

\begin{corollary}
    The number of isomorphism classes of indecomposable $B(n)$-modules is $n+(n+1)^2$. 
\end{corollary}
\begin{proof} Corollary \ref{cor : number of *classes of strings} tells us that there are exactly $(n+1)^2+1$ strings on $(Q(n), I(n))$ up to $*$-equivalence, which gives us $(n+1)^2+1$ indecomposable string $B(n)$-modules, up to isomorphism. If an indecomposable module is not a string, then  by the previous theorem it is isomorphic to $R(i)$ for some $i\in[0,n-2]$. It is immediate to see that $R(i)\simeq R(j)$ if and only if $i=j$. Therefore the number of indecomposable $B(n)$ modules which are not string is $\#[0,n-2]=n-1$.
\end{proof}

\subsection{Thin indecomposable $B(n)$-modules}In this section we recall the notion of thin module for a quiver with relations and we show that an  object $N$ is subobject of a non thin indecomposable $B(n)$-module $M$ if and only if there is a quotient of $M$ with same dimension vector as $N$. 

\begin{definition}Let $(Q,I)$ be a quiver with relations. We say that a module $M$ is \emph{thin} if the isomorphism classes of its decomposition factors are all distinct. 
\end{definition}
In our case, since $B(n)$ is basic, a module is thin if and only if its dimension vector has all entries $\leq 1$. Moreover, it is clear that if a module is indecomposable, then its support is connected.

\begin{lemma}\label{lemma : description_of_rooty}
    Let $M$ be a thin indecomposable $B(n)$-module and let $\underline{d}=(d_n, d_{n-1}, \ldots, d_0)$ be its dimension vector. Then there exist $0\leq i\leq j\leq n$ such that $d_k=1$ for any $k\in[i,j]$ and $d_k=0$ otherwise.
\end{lemma}
    

\begin{remark}
    Note that there is a unique string module which is not thin: the one corresponding to the walk $\beta_{n-1}\alpha_{n-1}$.
\end{remark}
\begin{corollary}\label{cor : number of bricks}
    The number of isomorphism classes of thin indecomposable $B(n)$-modules is $(n+1)^2$.
\end{corollary}

The following result will have as an application that we do not need non thin modules to determine the wall and chamber structure of $B(n)$.

\begin{lemma}\label{lemma : subobjs and quots of non thin }
    Let $M$ be an indecomposable non thin $B(n)$-module. Then $M$ is isomorphic to either $M(\beta_{n-1}\alpha_{n-1})$ or $R(i)$ for some $i\in [0,n-2]$. Moreover, 
    for any indecomposable subobject $N$ of $M$  there exists a quotient $R$ of $M$ such that $\dim N=\dim R$.
\end{lemma}
\begin{proof} Since $M$ is an indecomposable $B(n)$-module by Theorem \ref{thm:classification of indcps} $M$ is either a string module or is coming from a binomial relation in $I(n)$. 

Assume first that $M$ is a string module. By the proof of Lemma \ref{lemma : description_of_rooty} we know that all string modules but $M(\beta_{n-1}\alpha_{n-1})$ are thin.  Since $M$ is assumed not to be thin it must be  isomorphic to $M(\beta_{n-1}\alpha_{n-1})$. 
Via direct inspection we see that the only two proper submodules of $M(\beta_{n-1}\alpha_{n-1})$ are (isomorphic to) $S_{n}$ and $M(\beta_{n-1})$. On the other hand, its (proper) quotients are isomorphic to $S_n$ and $M(\alpha_{n-1})$. Since $\dim M(\beta_{n-1})=\dim M(\alpha_{n-1})$ the lemma is proven in the case $M\simeq M(\beta_{n-1}\alpha_{n-1})$.

We can hence assume that $M$ is not a string, so that $M\simeq R(i)$ for some $i\in[0,n-2]$. In this case by definition of $R(i)$ we have a non thin module. 
By direct inspection we see that  its proper subobjects are isomorphic to  $S_{i+1}$, $M(\alpha_{i+1})$, $M(\beta_i)$, $M(\alpha_{i+1}^*\beta_i)$, while its proper quotients are isomorphic to $S_{i+1},$ $M(\alpha_i),$ $M(\beta_{i+1}),$ $M(\beta_{i+1}\alpha_i^*)$.
As $\dim M(\alpha_j)=M(\beta_j)$ for any $j\in[0,n-1]$ and $\dim M(\beta_{i+1}\alpha_i^*)=\dim M(\alpha_{i+1}^*\beta_i)$ the claim follows.
\end{proof}

\begin{definition} Let $(Q,I)$ be a quiver with relations. We say that a module $M$ is a brick if its endomorphism ring is isomorphic to $\kk$.
\end{definition}

\begin{proposition}\label{prop:  brick IFF thin incecomposable}
    A $B(n)$-module is a brick if and only if it is a thin indecomposable module.
\end{proposition}
\begin{proof}
    It is well known that a thin indecomposable module is necessarily a brick. 
    
    Viceversa, assume that $M$ is a brick. Certainly, $M$ has to be indecomposable. Assume that $M$ is not thin, then by  the proof of Lemma \ref{lemma : subobjs and quots of non thin } we know that there exists an $i\in [0,n-2]$ such that $S_{i+1}$ is both a subobject and a quotient of $M$. If we denote by $f:S_{i+1}\rightarrow M$ the inclusion, and by $g:M\rightarrow S_{i+1}$ the projection to the quotient, then the composite $f\circ g$ is clearly not a multiple of the identity morphism $\textrm{Id}_M$. This contradicts the fact that $\textrm{End}(M)\simeq \kk$. 
   
\end{proof}

\section{Wall and chamber structure}\label{sec:WCstructure}
We recall here the definition of wall and chamber structure, following \cite[\S1.2, \S1.3]{WC_structure}. Since we only deal with our very special quiver with relations $(Q(n),I(n))$, we give the various  definitions adapted to our case and do not discuss the more general setting.

Firstly, we denote by $\langle\cdot, \cdot\rangle$ the usual scalar product on $\mathbb{R}^{n+1}$. We start by recalling King's notion of $v$-semistable modules for an algebra with relations \cite{king}.

\begin{definition}\label{defn : semistability} Let $v\in \mathbb{R}^{n+1}$. 
We say that a $(Q(n), I(n))$-module $M$ is \emph{$v$-semistable}  if $\langle v, \dim M\rangle=0$ and $\langle v, \dim N \rangle\leq 0$ for any non zero proper subobject $N$ of $M$. Equivalently, we say that $M$ is \emph{$v$-semistable} if $\langle v, \dim M\rangle=0$ and $\langle v, \dim R \rangle\geq 0$ for any  non zero proper quotient $R$ of $M$.
\end{definition}

\begin{remark}\label{rmk : stability conds indecp submods}Let $M$ be a non zero $(Q(n), I(n))$-module and let $v\in \mathbb{R}^{n+1}$.
    Observe that if $N$, $N'$ are subobjects of $M$ such that $N\cap N'=(0)$,  $\langle v, \dim N\rangle\leq 0$ and $\langle v, \dim N'\rangle\leq 0$ then $N\oplus N'$ is a subobject of $M$ and automatically $\langle v, N\oplus N'\rangle \leq 0$. Therefore it is enough to consider indecomposable subobjects in Definition \ref{defn : semistability}.
\end{remark}

\begin{definition}(\cite[Definition 11]{WC_structure})\begin{enumerate}
    \item Let $M$ be a non zero $B(n)$-module. Its \emph{stability space} is
    \[\mathcal{D}(M)=\{v\in\mathbb{R}^{n+1}\mid \ M \hbox{ is $v$-semistable}\}.\]
    \item We say that the stability space $\mathcal{D}(M)$ of a non-zero module is a \emph{wall} if it has codimension 1.
    \item The \emph{chambers} are open connected components of \[\mathbb{R}^{n+1}\setminus \overline{\bigcup \mathcal{D}(M)}, \]
    where the union runs over the set of indecomposable $B(n)$-modules.
    \item The \emph{wall and chamber structure for $B(n)$} is the combination of all the walls ($\mathcal{D}(M)$ for indecomposable modules $M$) and all the chambers.
\end{enumerate}
    \end{definition}
   
    \begin{remark}
    In \cite{WC_structure} to define the chambers the complement is taken with respect to the closure of the union of all non-zero modules, but then \cite[Proposition 12]{WC_structure} shows that it is enough to consider the indecomposable ones. 
    \end{remark}

    \begin{example}\label{exple : WC r=1} In the case $n=1$ we have five indecomposable. We keep the same notation as in Example \ref{exple:indcps for r=1} and describe the wall and chamber structure of $B(1)$. Determining stability space of simple modules is immediate, as by definition they do not have any proper submodule:
    \[
    \mathcal{D}(S_0)=\{v\in \mathbb{R}^2\mid \langle v, (0,1)\rangle=0\}=\hbox{$x$ axis},
    \]
    \[
    \mathcal{D}(S_1)=\{v\in \mathbb{R}^2\mid \langle v, (1,0)\rangle=0\}=\hbox{$y$ axis}.
    \]
    We are left with other three indecomposable modules: $M(\alpha_0)$ and $M(\beta_0)$ (dual to each other) and the (injective-projective) indecomposable object $M(\beta_0\alpha_0)$. Observe that the only proper submodule of $M(\alpha_0)$ is $S_0$, which is also the unique proper quotient of $M(\beta_0)$, so that
    \[
\mathcal{D}(M(\alpha_0))=\left\{v\in \mathbb{R}^2\ \middle\vert 
\begin{array}{c}
   \langle v, (1,1)\rangle=0,  \\
    \langle v, (0,1)\rangle\leq  0
\end{array}
\right\}=\left\{(t,-t)\mid t\in \mathbb{R}_{\geq 0}
\right\},
    \]
     \[
\mathcal{D}(M(\beta_0))=\left\{v\in \mathbb{R}^2\ \middle\vert
\begin{array}{c}
   \langle v, (1,1)\rangle=0,  \\
    \langle v, (0,1)\rangle\geq  0
\end{array}
\right\}=\left\{(t,-t)\mid t\in \mathbb{R}_{\leq 0}
\right\}.
    \]
    Finally, $M(\beta_0\alpha_0)$ has  only two proper subobjects, which are  $M(\beta_0)$ and $S_0$ ad hence its stability space is
\[
\mathcal{D}(M(\beta_0\alpha_0))= \left\{v\in \mathbb{R}^2\ \middle\vert
\begin{array}{c}
   \langle v, (2,1))\rangle=0, \\ \langle v, (1,1)\rangle\leq 0,  \\
    \langle v, (0,1)\rangle\leq  0
\end{array}
\right\}=\left\{(0,0)
\right\}.
\]
The wall and chamber structure in this case is represented in Figure \ref{fig:WCB1}. If we compare our picture with \cite[Figure 2]{MR2739061} we notice that the wall and chamber structure for $B(1)$ encodes the whole wall crossing phenomena in Bridgeland's stability space of the bounded derived category of constructible sheaves on $\mathbb{P}^1$.
\begin{figure}[ht!]
\centering
\begin{tikzpicture}
  \draw[gray, thick] (-3, 0) -- (3, 0) 
  node[right] 
  {\textcolor{gray}{$\mathcal{D}(S_0)$}}
  ;
  \draw[blue, thick] (0, -3) -- (0, 3) 
   node[above] 
  {\textcolor{blue}{$\mathcal{D}(S_1)$}}
  ;
  \draw[orange, thick] (0,0) -- (3,-3) 
  node[right] 
  {\textcolor{orange}{$\mathcal{D}(M(\alpha_0))$}}
  ;
  \draw[purple, thick] (0,0) -- (-3,3) 
  node[left] 
  {\textcolor{purple}{$\mathcal{D}(M(\beta_0))$}}
  ;
  \filldraw[black] (0,0) circle (2pt); 
  \node[-] at (1.2,0.4){$\mathcal{D}(M(\beta_0\alpha_0))$};
\end{tikzpicture}
\caption{Wall and chamber structure of $B(1)$}\label{fig:WCB1}
\end{figure}
    \end{example}

\begin{lemma}\label{lemma : stab space proper subobj and proper quot}
    Let $M$ be a $(Q(n), I(n))$-module and assume that there exists a proper nonzero submodule $N$ and a proper nonzero quotient $R$ such that $\dim N=\dim R$. Thus, $\mathcal{D}(M)\subseteq\mathcal{D}(N)$. 
\end{lemma}
\begin{proof}
    Recall that 
    \begin{align*}
\mathcal{D}(M)&=\left\{v\in\mathbb{R}^{n+1}\middle\vert \begin{array}{c}\langle v,\dim M\rangle=0, \\
\langle v,
\dim N'\rangle\leq 0
\hbox{ for any subobj } N'\subset M\end{array}
\right\}\\
&=\left\{v\in\mathbb{R}^{n+1}\middle\vert \begin{array}{c}\langle v,\dim M\rangle=0, \\
\langle v,
\dim R'\rangle\geq 0
\hbox{ for any quotient } M\twoheadrightarrow R'.  \end{array}
\right\}.
    \end{align*}

Let $v\in \mathcal{D}(M)$. Since $\dim N=\dim R$,  then it must hold $\langle v, \dim N\rangle=0$. Moreover, N being a suboject of $M$, every subobject of $N$ is also a subobject of $M$, and hence $\langle v, \dim N'\rangle\leq 0$ for any subobject $N'\subseteq N$. Thus we conclude that $\mathcal{D}(M)\subseteq\mathcal{D}(N)$.    
\end{proof}

The following lemma can be also found in \cite{asai}.

\begin{lemma}\label{lemma : stab spaces of non thin do not give new walls}
    Let $M$ be an indecomposable $B(n)$-module. Then there exists a thin  indecomposable  $L$ such that $\mathcal{D}(M)\subseteq  \mathcal{D}(L)$.
\end{lemma}
\begin{proof}If $M$ is thin the statement is trivial, so that we can assume that $M$ is not thin. In this case, by the proof of Lemma \ref{lemma : subobjs and quots of non thin }
 there exists an $i\in [0,n-1]$ such that $S_{i+1}$ is both a submodule and a quotient of $M$. Thus by Lemma \ref{lemma : stab space proper subobj and proper quot} we have that $\mathcal{D}(M)\subset \mathcal{D}(S_i)$.
\end{proof}

The following proposition is now an immediate consequence of Lemma \ref{lemma : stab spaces of non thin do not give new walls} and the definition of wall and chamber structure.

\begin{proposition}\label{theorem : rooty_walls}
    The wall and chamber structure of $B(n)$ is uniquely determined by conditions coming from thin indecomposables.
\end{proposition}

\begin{remark}
   Recall that by Proposition \ref{prop:  brick IFF thin incecomposable} a $B(n)$-module is a brick if and only if is a thin indecomposable. Therefore Proposition \ref{theorem : rooty_walls} could also be deduced from  \cite[Proposition 2.7]{asai}.
\end{remark}

\begin{corollary}\label{cor:number_of_walls}
    The number of walls in the wall and chamber structure of $B(n)$ is $(n+1)^2$.
\end{corollary}
\begin{proof}
    This is an immediate consequence of Lemma \ref{lemma : stab spaces of non thin do not give new walls} and Corollary \ref{cor : number of bricks}. 
\end{proof}

In the proof of Lemma \ref{lemma : subobjs and quots of non thin } we have classified the submodules of any non thin module. This allows us to find explicit inequality description for the stability spaces of indecomposable non thin modules. Notice that by the previous theorem this is not needed for the wall and chamber structure of $B(n)$, but we decided to include it in our paper for completeness since it is easy to deduce from what we have already seen. 

\begin{proposition}\label{prop : walls for non thin}The stability spaces for the non thin indecomposable modules of $B(n)$ are as follows:
\[
\mathcal{D}(M(\beta_{n-1}\alpha_{n-1}))=\left\{v=(v_n, \ldots, v_0)\in\mathbb{R}^{n+1}\mid v_n=v_{n-1}=0\right\},
\]
    \[
    \mathcal{D}(R(i))=\left\{v=(v_n, \ldots, v_0)\in\mathbb{R}^{n+1}\mid v_i=v_{i+1}=v_{i+2}=0\right\}, \quad i\in[0,n-2].
    \]
    In particular, if $M$ is not thin then $\mathcal{D}(M)$ is not a wall. 
\end{proposition}
\begin{proof}First let $M\simeq M(\beta_{n-1}\alpha_{n-1})$. If  $v\in\mathcal{D}(M)$ then $\langle v,\dim M\rangle=0$, that is $v_{n-1}+2v_{n}=0$. But there is a subobject and a quotient of $M$ which are isomorphic to $S_n$ (cf. proof of Lemma \ref{lemma : subobjs and quots of non thin }). Thus $v_n=0$ and hence also $v_{n-1}=0$. This is a also sufficient condition for $v$ to belong to $\mathcal{D}(M)$ as all submodules of $M$ would have dimension vector whose entries different from the $n$-th and $n-1$-st are zero, so that we do not have any condition on $v_j$ for $j\neq n, n-1$.

Assume now that $M\simeq R(i)$ for some $i\in[0,n-2]$. Then all submodules and quotients of $M$ have dimension vector whose non zero entries are indexed by some subset of $[i,i+2]$. This means that the conditions on a vector $v\in\mathbb{R}^{n+1}$ to belong to $\mathcal{D}(M)$ involve only $v_{i}, v_{i+1}, v_{i+2}$ while all the others are free. If $v$ is in the stability space of $M$ has to verify 
\begin{equation}\label{eqn : stab spaces non thin}
    v_i+2v_{i+1}+v_{i+2}=0.
\end{equation} Moreover  (cf. proof of lemma \ref{lemma : subobjs and quots of non thin }) $M$ has submodules (among others) isomorphic to $S_{i+1}$, $M(\beta_i)$ and has a subquotients $S_{i+1}$ and $M(\alpha_i)$. We deduce that
\[
\langle v, \dim S_{i+1}\rangle=v_{i+1}=0, \quad \langle v, \dim M(\beta_i)\rangle=v_i+v_{i+1}=0.
\]
The first equality together with the second one tells us that $v_{i+1}=v_i=0$ and this combined with \eqref{eqn : stab spaces non thin} tells us that also $v_{i+2}=0$. By the above considerations these are also sufficient conditions.
\end{proof}

\begin{example}
We apply Proposition \ref{theorem : rooty_walls}  to determine the wall and chamber structure in the $n=2$ case. In principle, Corollary \ref{cor : number of *classes of strings} tells us that we need to compute $3^2+2=11$ stability spaces, but by the previous theorem we know that it is enough to determine 9 of them. Note that there are three stability spaces given by simples and other three pairs of stability spaces given by (dual) thin indecomposables (this is in fact a general feature, as explained in  the proof of the next Theorem \ref{theorem:explicit_descr_stab_spaces}).

 Recall that the thin indecomposable $B(2)$-modules are parameterized by the set $\mathscr{S}'=[0,2]\cup\left\{(a,b,\eta)\mid 0\leq a<b\leq 2,\   \eta\in\{\pm 1\}\right\}$. We denote by $M(c)$ the indecomposable module corresponding to $c\in \mathscr{S}'$.
 
\[
\mathcal{D}(M(0))= xy\hbox{-hyperplane}, \quad \mathcal{D}(M(1))= xz\hbox{-hyperplane}, \quad \mathcal{D}(M(2))= yz\hbox{-hyperplane}, 
\]
\[
\mathcal{D}(M(0,1,-1))= \mathcal{D}(M(\alpha_0))= \left\{
(x,y,z)\in\mathbb{R
}^3 \middle\vert \begin{array}{l}
y+z=0,  \\
z\leq 0. 
\end{array}\right\}
\]
\[
\mathcal{D}(M(0,1,1))= \mathcal{D}(M(\beta_0)) = \left\{
(x,y,z)\in\mathbb{R
}^3\middle\vert \begin{array}{l}
y+z=0,  \\
z\geq 0. 
\end{array}\right\}
\]
\[
\mathcal{D}(M(1,2,-1))= \mathcal{D}(M(\alpha_1)) =\left\{
(x,y,z)\in\mathbb{R
}^3\middle\vert \begin{array}{l}
x+y=0,  \\
y\leq 0. 
\end{array}\right\},
\]
\[
\mathcal{D}(M(1,2,1))= \mathcal{D}(M(\beta_1)) =\left\{
(x,y,z)\in\mathbb{R
}^3\middle\vert \begin{array}{l}
x+y=0,  \\
y\geq 0. 
\end{array}\right\},
\]
\[
\mathcal{D}(M(0,2,-1))= \mathcal{D}(M(\alpha_0\beta_1^*)) = \left\{
(x,y,z)\in\mathbb{R
}^3\middle\vert \begin{array}{l}
x+y+z=0,  \\
x\leq 0, \ z\leq 0
\end{array}\right\},
\]
\[
\mathcal{D}(M(0,2,1))= \mathcal{D}(M(\alpha_1^*\beta_0)) = \left\{
(x,y,z)\in\mathbb{R
}^3\middle\vert \begin{array}{l}
x+y+z=0,  \\
x\geq 0, \ z\geq 0
\end{array}\right\}.
\]
 Figure \ref{fig:WCB(2)} illustrates the wall and chamber structure of $B(2)$.
 
\begin{figure}[H]
    \centering
    \includegraphics[width=15cm]{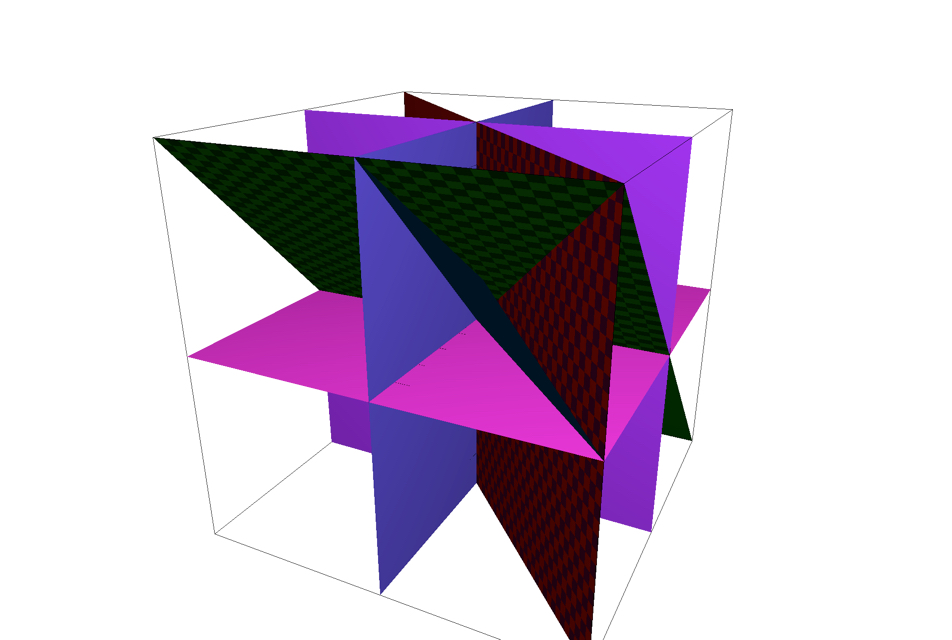}
    \caption{Wall and chamber structure of $B(2)$.}
    \label{fig:WCB(2)}
\end{figure}

In Figure \ref{fig:WCB(2)}, the three hyperplanes $\mathcal{D}(M(0)), \mathcal{D}(M(1))$ and $\mathcal{D}(M(2))$ appear in pink, purple and blue respectively. Moreover, $\mathcal{D}(M(0,1,1))$ and $\mathcal{D}(M(0,1,-1))$ appear in black and green, $\mathcal{D}(M(1,2,-1))$ and $\mathcal{D}(M(1,2,1))$ appear in black and red, and finally $\mathcal{D}(M(0,2,-1))$ and $\mathcal{D}(M(0,2,1))$ are in dark blue. 
\end{example}

The following theorem generalises the previous example, providing an explicit description of all the relevant stability spaces:
\begin{theorem}\label{theorem:explicit_descr_stab_spaces}
    We have the following description of the stability spaces:
\begin{enumerate}
        \item $\mathcal{D}(M(\epsilon_i))=\{(v_n, v_{n-1} \ldots, v_{0})\in\mathbb{R}^{n+1}\mid v_{i}=0 \}$,
        \item if $c=(a,b,-1)$ for some $0\leq a<b\leq n$, then 
        \[
\mathcal{D}(M(a,b,-1))=\left\{(v_n, v_{n-1} \ldots, v_{0})\in\mathbb{R}^{n+1}\middle\vert
   \begin{array}{l}
   v_a+v_{a+1}+\ldots + v_b=0, \\
   \sum_{i=a'}^{b'}v_i\leq 0, \begin{array}{l}
     [a',b']\subseteq [a,b], \ a'\equiv_2 a,     \\
         (b'=b \hbox{ or }b'\equiv_2 a)
   \end{array}
   \end{array}\right\},
   \]
\item if $c=(a,b,1)$ for some $0\leq a<b\leq n$, then 
        \[
\mathcal{D}(M(a,b,1))=\left\{(v_n, v_{n-1} \ldots, v_{0})\in\mathbb{R}^{n+1}\middle\vert
   \begin{array}{l}
   v_a+v_{a+1}+\ldots + v_b=0, \\
   \sum_{i=a'}^{b'}v_i\geq 0, \begin{array}{l}
     [a',b']\subseteq [a,b], \ a'\equiv_2 a,     \\
         (b'=b \hbox{ or }b'\equiv_2 a)
   \end{array}
   \end{array}\right\},
   \]    \end{enumerate}
   where $x\equiv_2 y$ indicates that $x$ and $y$ have the same parity.
\end{theorem}
\begin{proof}
    Claim \emph{(1)} follows immediately from the fact that $M(\epsilon_i)$ is nothing but a simple module $\simeq S_i$ whose dimension vector has only one nonzero entry in position $i$.

 Let $M\simeq M(a,b,\eta)$. Observe that any submodule or quotient of a thin module is itself thin, therefore a submodule $N$, resp.  a quotient $R$, of $M$ is uniquely determined by the set $i\in [a,b]$ such that $N_i\neq 0$, resp. $R_i\neq 0$. Let $w=\gamma_m\ldots \gamma_1$ be the string $\psi((a,b,\eta))$ (see proof of Lemma \ref{ lemma : parameterisation of strings}).  By definition of string representations it is clear that $N$ is a submodule if and only if whenever there is an $i\in[n]$ such that $N_i\neq (0)$ then the following two conditions hold:
 \begin{itemize}
     \item there exists  $j\in[0,m]$ such that $f_w(j)=i$,
     \item if there is a $k\in[m]$ such that $s(\Tilde{\gamma_k})=i$ then $N_{t(\Tilde{\gamma_k})}\neq 0$.
 \end{itemize}
 Dually, $R$ is a quotient if and only if whenever there is an $i\in[n]$ such that $R_i\neq (0)$ then the following two conditions hold:
 \begin{itemize}
     \item there exists  $j\in[0,m]$ such that $f_w(j)=i$,
     \item if there is a $k\in[m]$ such that $t(\Tilde{\gamma_k})=i$ then $R_{s(\Tilde{\gamma_k})}\neq 0$.
 \end{itemize}
We also observe that if $M_i\neq (0)$ then either $i$ is source for all adjacent $\Tilde{\gamma}$ or is target for all adjacent $\Tilde{\gamma}$. Therefore \emph{(2)} and \emph{(3)} are  equivalent: $\dim M(a,b,-1)=\dim M(a,b,1)$ and if follows from what we just discussed that a $B(n)$-module $N$ is a subobject of $M(a,b,-1)$ if and only if it is a quotient of $M(a,b,1)$. This implies that the defining inequalities of $\mathcal{D}(M(a,b,1))$ are obtained by reversing the ones of $\mathcal{D}(M(a,b,-1))$ and viceversa. 

    Thus we are reduced to show \emph{(2)}. Let   $M=M(a,b,-1)$. Since $\dim M_i=1$ for any $i\in[a,b]$ and $=0$ otherwise, we have that $\langle (v_n, v_{n-1}, \ldots, v_0),\dim M\rangle=0$ if and only if $\sum_{i=a}^b v_i=0$. Note that $S_i$ is a submodule if and only if there exists a $k\in[m]$ such that $t(\Tilde{\gamma_k})=i$ (and hence $i\in[a,b]$). Let $w=\gamma_m\ldots \gamma_1=\psi((a,b,-1))$. In this case we have $\gamma_m=\alpha_a^*$ and hence $a$ is a target of $\Tilde{\alpha_a^*}$. Therefore the targets are the vertices labelled by elements $c\equiv_2 a$. We conclude that $S_i$ is a submodule of $M$ if and only if $i\equiv_2 a$. By the classification of indecomposable thin modules we also know that all indecomposable submodules of $M$ corresponding to alternating walks have to be parameterized by a subset of
    $\left\{(a', b', \eta')\mid a\leq a'<b'\leq b, \ \eta'\in\{\pm 1\}\right\}$. By the same reasoning as  before, we conclude that $a'\equiv_2 a$. If $b'\neq b$, once again it has to satisfy $b'\equiv_2 a$, otherwise it would be source of some honest arrow $\Tilde{\gamma_k}$. Instead, the vertex $b=f_w(m)$ cannot be source of any arrow not contained in $[a',b]$ and hence $M(a',b,-1)$ is (isomorphic to) a submodule of $M$ for any $a'\equiv_2 a$. The claim now follows by taking dimension vectors and calculating the scalar products with $v$.
\end{proof}

\begin{remark}
    By \cite{asai} and \cite{BST19} the chambers obtained by removing the walls described in Theorem \ref{theorem:explicit_descr_stab_spaces} parameterize isomorphism classes of 2-silting complexes for $B(n)$. Moreover, if a pair of 2-silting complexes correspond to adjacent chambers then there is a mutation taking one of the 2-silting complexes to the other. 
    
    For instance, in the case $n=2$ we can exploit Theorem \ref{theorem:explicit_descr_stab_spaces} (and Figure \ref{fig:WCB(2)}) to see that there are exactly 20 chambers and hence 20 (isomorphism classes of) 2-silting complexes. To obtain the mutation graph it is hence enough to draw a vertex in the interior of any chamber and an edge between vertices of chambers having a wall in common. It would be interesting to apply our explicit description to deduce information on $B(n)$-silting complexes and their properties for general $n$.
\end{remark}

\section{Perverse sheaves on the complex projective space and $B(n)$}\label{sec:perverse}

\subsection{Generalities on perverse sheaves}

In this section we recall some definitions and important facts on perverse sheaves, see \cite{BBD, gm2}

\subsubsection{Topologically stratified spaces}
 We recall the inductive definition of a topologically stratified space in the sense of \cite{gm2}, see also \cite[Definition 4.1.1]{KW}.
 
 A 0-dimensional topologically stratified space is a discrete union of points. Let $d\geq1$. A $d$-dimensional topologically stratified space $X$ is a paracompact Hausdorff topological space with a finite filtration by closed subset $X=X_d\supset \ldots \supset X_0\supset X_{-1}=\emptyset$ such that $X_i \smallsetminus X_{i-1}$ is a (possibly empty) $i$-dimensional topological manifold, and such that each $x\in X_i \smallsetminus X_{i-1}$ has an open neighborhood filtration-preserving homeomorphic to $\R^i\times C(L_i)$ where 
 \begin{itemize}
     \item $L_i$ is some compact $(d-i-1)$-topologically stratified space called {\em link} of $X_i\smallsetminus X_{i-1}$ at $x$.
     \item $C(L_i) = L_i \times [0, 1)/L_i \times\{0\}$ denotes the {\em open cone on $L_i$} with the induced filtration by the vertex and the subsets $L_k \times[0, 1)/L_k \times\{0\}$, with $k\leq i$.
 \end{itemize}
 The (non empty) connected components of $X_{i}\smallsetminus X_{i-1}$ are called {\em strata} of $X$.

\subsubsection{The constructible derived category}

Let $\kk$ be a field. Let $X$ be a topologically stratified space and let $\der{b}{X}$ be the bounded derived category of sheaves of $\kk$-vector spaces on $X$. The {\em constructible derived category of $X$}, which we will denote by $\constr{c}{X}$, is the full subcategory of $\der{b}{X}$ with objects the complexes whose cohomology sheaves are locally-constant when restricted to each stratum of $X$. Note that $\constr{c}{X}$ is a triangulated category.

Let $Z$ be a closed union of strata of $X$. The complementary maps  $\imath\colon Z\to X$ and $\jmath\colon U=X\smallsetminus Z\to X$ give rise to a triangulated recollement, see for instance \cite[\S1.4.1]{BBD}, of the form
\begin{equation}\begin{tikzcd} [row sep={0.6cm},column sep={0.6cm}]\label{equation:6ff}
\constr{c}{Z}\ar[rr,"\imath_*=\imath_!"]&&\constr{c}{X}\ar[ll,"\imath^!",bend left=35]\ar[ll,"\imath^*", swap,bend right=50]\ar[rr,"\jmath^*=\jmath^!"]&&\constr{c}{U}\ar[ll,"\jmath_*",bend left=35]\ar[ll,"\jmath_!", swap,bend right=50]\end{tikzcd}.\end{equation}

In particular, $(\imath^*,\imath_*), (\imath_*,\imath^!), (\jmath_!,\jmath^*)$ and $(\jmath^*,\jmath_*)$ are adjoint pairs.

\subsubsection{Perverse t-structures}

Let $X$ be a topologically stratified space with strata $S_i$. A \emph{perversity $p$ on $X$} is a $\Z$-valued function from the set of strata of $X$.

Let $\imath_S\colon S\to X$ be the inclusion of a stratum into $X$. The pair of subcategories
\begin{equation}\label{perv_t_str}
\begin{split}
    {^p{D^{\leq0}}}&=\{\sh{A}\in\constr{c}{X}\mid \sh{H}^k(\imath_S^*\sh{A})=0 \quad k>p(S)\} \\
    {^p{D^{\geq0}}}&=\{\sh{A}\in\constr{c}{X}\mid \sh{H}^k(\imath_S^!\sh{A})=0 \quad k<p(S)\}
\end{split},
\end{equation}
where $\sh{H}^k(\sh{E})$ denotes the cohomology sheaf of the complex of sheaves $\sh{E}$, is a bounded t-structure on $\constr{c}{X}$ for any perversity $p$ on $X$, see \cite[2.1.4]{BBD}. The t-structure $({^p{D^{\leq0}}}, {^p{D^{\geq0}}})$ is called the \emph{$p$-perverse t-structure on $\constr{c}{X}$}. The category of $p$-perverse sheaves on $X$ is defined as the {\em heart} of such t-structure, that is $\perv{p}{X}={^p{D^{\leq0}}}\cap{^p{D^{\geq0}}}$. It follows from the general theory of t-structure that $\perv{p}{X}$ is an abelian subcategory of $\constr{c}{X}$. The triangulated recollement (\ref{equation:6ff}) descend to an abelian recollement of categories of perverse sheaves, with \emph{perverse functors} defined by setting $^pF={^pH^0}\circ F \circ \epsilon$, where $F\in\{\imath^*,\imath_*,\imath^!,\jmath_!,\jmath^*,\jmath_* \}$, $^pH^0\colon\constr{c}{Y}\to\perv{p}{Y}$ for $Y\in\{Z,X,U\}$ is the cohomological functor left inverse to the inclusion $\epsilon\colon\perv{p}{Y}\to\constr{c}{Y}$, see \cite[\S1.3.6]{BBD}.

\subsubsection{Perverse sheaves as representations of a quiver with relations}

Under suitable topological assumptions on $X$, if $\kk$ is algebraically closed we have an exact equivalence between categories of perverse sheaves and representations of a quiver with relations.

\begin{theorem}\cite[Corollary 5.3 and \S5.1]{MR4349392} \label{thm:rep_perv_sh}
Let $\kk$ be an algebraically closed field and $X$ a topologically stratified space with finitely many strata, each with finite fundamental group. For any perversity $p$ on $X$ the category $\perv{p}{X}$ is equivalent to a category of representations of a quiver with relations.
\end{theorem}

\subsection{Perverse sheaves on the complex projective space}

Let $n\geq1$ and consider the complex projective space $X=\PP^n$ with the filtration
\[
X=X_{2n}\supset X_{2n-1}=X_{2n-2}=\PP^{n-1}\supset\ldots\supset X_3=X_2=\PP^1\supset X_1=X_0=\pt\supset X_{-1}=\emptyset,
\]

so that there are $n+1$ (non empty) strata  indexed by their complex dimension
\[
S_i=\PP^i\smallsetminus\PP^{i-1}\cong\C^i \quad i\in[0,n].
\]
Note that all the strata $S_i$ of $X$ are such that $\pi_1(S_i)$ is trivial for all $i\in[0,n]$.

\begin{remark}
    It might seem odd that half of the strata are empty, but this is a very classical stratification of the complex projective space when we see it as a partial flag variety. Indeed, each non-empty stratum is a Schubert cell, that is an orbit for the action of a Borel subgroup of $\mathrm{SL}_{n+1}$.
\end{remark}

Let $m$ be the middle perversity on $X$ (that is $m(S_i)=-i$ for all $i\in[0,n]$) and $\kk$ be an algebraically closed field. 

By Theorem \ref{thm:rep_perv_sh}, there is an exact equivalence between the category of $m$-perverse sheaves on $X$ and the category of representations of the quiver $(Q(n),I(n))$ defined in \S\ref{section:background}, that is
\[
\perv{m}{\PP^n}\simeq \mod{\kk Q(n)/I(n)}.
\]

The category of middle perversity perverse sheaves on $X=\PP^n$ constructible with respect to the Schubert stratification is equivalent to a regular block of the parabolic category $\mathcal{O}$ for $\mathfrak{sl}_{n+1}$ and (maximal) parabolic $\mathfrak{q}\simeq \mathfrak{sl}_{n}$ as for \cite[Proposition 3.5.1]{MR1322847}. On the other hand such a block is equivalent to  the category of finite dimensional modules for 
$\kk Q(n)/I(n)$ we investigated in this paper (see \cite[Example 1.1]{stroppelTQFT}).

\subsection{Back to Bridgeland stability conditions}

Recall that we are aiming at getting a better understanding of  Bridgeland's space of stability conditions $\stab{\constr{c}{\PP^n}}$ on the triangulated category $\constr{c}{\PP^n}$. 

By \cite[\S1.5]{MR2119139}, the category of middle perverse sheaves on $\PP^n$ forms a faithful heart since we are stratifying the projective space by affine subvarieties whose cohomology is concentrated in one degree. This means that if we take the derived category of middle perverse sheaves on $\PP^n$ we get back to $\constr{c}{\PP^n}$.

Let $K(\constr{c}{X})$ denote the Grothendieck group of the constructible derived category of $X$. We now recall the definition of Bridgeland stability conditions, see \cite[Definition 1.1]{bridgeland07}. A {\em stability condition on $\constr{c}{X}$} is a pair $(Z,\sh{P})$ where:
\begin{enumerate}
    \item $Z\colon K(\constr{c}{X})\to\C$ is a group homomorphism,
    \item $\sh{P}=\left( \sh{P}(\phi) \right)_{\phi\in\R}$ with $\sh{P}(\phi)$ is a full additive subcategory of $\constr{c}{X}$,
\end{enumerate}
which satisfy the following axioms:
\begin{enumerate}[i)]
    \item if $E\in\sh{P}(\phi)$, then $Z(E)=r(E)\exp(i\pi\phi)$ for some $r(E)\in\R_{>0}$,
    \item for all $\phi\in\R$ we have $\sh{P}(\phi+1)=\sh{P}(\phi)[1]$,
    \item if $\phi_1>\phi_2$ and $A_j\in\sh{P}(\phi_j)$ for $j=1,2$ then $\Mor{\constr{c}{X}}{A_1}{A_2}=0$,
    \item for each non zero object $E\in\constr{c}{X}$ there are a finite sequence of real numbers $\phi_1>\phi_2>\ldots>\phi_n$ and a collection of triangles
    \begin{equation*}
    \begin{tikzcd}[column sep={0.5cm}]
    0=E_0\ar[rr] && E_1\ar[rr]\ar[dl]  && E_2\ar[rr]\ar[dl] && \ldots\ar[rr] && E_{n-1}\ar[rr] && E_n=E \ar[dl] \\ 
    & A_1\ar[ul,dashed] && A_2\ar[ul,dashed] && && && A_n\ar[ul,dashed]
    \end{tikzcd}
    \end{equation*}
    with objects $A_j\in\sh{P}(\phi_j)$ for all $j$, where a dashed arrow $A\dashrightarrow B$ represents a morphism $A\to B[1]$.
\end{enumerate}

A stability condition $(Z,\sh{P})$ determines a bounded t-structure on $\constr{c}{X}$ with heart the extension-closed subcategory generated by objects in $\sh{P}(\phi)$ with $\phi\in(0,1]$. Viceversa, a bounded t-structure on $\constr{c}{X}$ together with a stability function on the heart with the abelian analogue of axiom iv) above determines a stability condition on $\constr{c}{X}$ by \cite[Proposition 5.3]{bridgeland07}.

A stability condition is {\em locally-finite}, see \cite[Definition 5.7]{bridgeland07}, if there exists a real number $\epsilon>0$ such that for all $t\in\R$ the category obtained by taking the extension-closure:
\[
\langle E\in\sh{P}(\phi) \mid \phi\in(t-\epsilon,t+\epsilon)\rangle
\]
has finite length.

The set $\stab{\constr{c}{X}}$ of locally finite stability conditions on $\constr{c}{X}$ is a complex manifold by \cite[Theorem 1.2]{bridgeland07}. Moreover, 
we know that the wall and chamber structure of the category $\mathcal{A}$ of modules  over a finite dimensional algebra $A$  encodes to the wall and chamber structure of an open subset of $\stab{\der{b}{\mathcal{A}}}$ in the following sense. Let $\mathcal{C}(\mathcal{A})$ be the subvariety of $\stab{\der{b}{\mathcal{A}}}$ of stability conditions nearby $\mathcal{A}$, that are the points $(Z, \mathcal{P})$ such that $\mathcal{A}\subseteq \bigcup_{\phi\in(-1,1)}\mathcal{P}(\phi)$. By \cite[Lemma 7.3]{bridgelandWC},  $\mathcal{C}(\mathcal{A})$ is an open submanifold of $\stab{\der{b}{\mathcal{A}}}$ and, by \cite[Proposition 7.4]{bridgelandWC} there is a surjective map $F: \mathcal{C}(\mathcal{A})\rightarrow \mathbb{R}^{\textrm{rk}(K(\mathcal{A}))}$ such that the preimage under $F$ of a wall, resp. of  chamber, in the wall and chamber structure of $A$ is a wall, resp. of a chamber of $\mathcal{C}(\mathcal{A})$.  Since $\mod {B(n)}$ is equivalent to $\perv{m}{\PP^n}$, the wall and chamber structure of $B(n)$ allows us to understand the wall crossing phenomena on the manifold of  stability conditions nearby $\perv{m}{\PP^n}$ in 
 $\stab{{ \der{b}{\mod {B(n)}}}}= \stab{\der{b}{\perv{m}{\PP^n}}}=\stab{\constr{c}{\PP^n}}$.

\subsection{Geometric interpretation of the thin $B(n)$-modules}

Let us denote the inclusion of a stratum into $X$ by $\imath_{k}\colon S_k\hookrightarrow X$. The simple perverse sheaves (with respect to the middle perversity) are in bijection with strata of $X$: they can be expressed using the intermediate extension functor ${^m\imath_k}_{!*}\coloneqq \im({^m\imath_k}_!\to{^m\imath_k}_*)$, see \cite[1.4.22]{BBD} or \cite[\S3.3]{Achar}, as
\[
\ic{}{r}={^m\imath_{n-r}}_{!*}\kk_{\C^{n-r}}[n-r]=\kk_{\PP^{n-r}}[n-r] \quad 0\leq r \leq n,
\]
where in the last expression we omit an (exact) extension by zero. 



Let $k,l\in\mathbb{N}$. For $n\geq k > l \geq 0$ we set $\imath_{k,l}\colon S_l\to \cup_{r=l}^k S_r$ and $\jmath_{k,l}\colon\cup_{r=l+1}^k S_r\to \cup_{r=l}^k S_r$. 

\begin{remark}
Note that $\imath_{k,l}$ is a closed inclusion, so that ${\imath_{k,l}}_!={\imath_{k,l}}_*$. Be aware that, despite of the similar notation, the inclusion of a single stratum $\imath_k$ is in general neither open nor closed (it is closed only when $k=0$ and it is open only when $k=n$). Instead, the morphism $\jmath_{k,l}$ is always open. Moreover, $j_{l+1,l}$ is affine and therefore $^m{j_{l+1,l}}=j_{l+1,l}$ by \cite[\S5.3]{DCM}.
\end{remark}

Let $0\leq a < b \leq n$, we define
\begin{equation}\label{thin_perverse1}
\mathcal{M}(a,b,1)={^m\jmath_{n,0}}_{!*}{^m\jmath_{n,1}}_{!*}\ldots{^m\jmath_{n,n-b-1}}_{!*}{^m\jmath_{n,n-b}}_?\ldots {^m\jmath_{n,n-a-2}}_*{^m\jmath_{n,n-a-1}}_!{\imath_{n,n-a}}_*\kk_{\mathbb{C}^{n-a}[n-a]}
\end{equation}
where $?=\begin{cases}
    * \ \hbox{if}\ a \equiv_2 b,\\ ! \ \hbox{if}\ a \not\equiv_2 b.
\end{cases}$, and 
\begin{equation}\label{thin_perverse2}
\mathcal{M}(a,b,-1)={^m\jmath_{n,0}}_{!*}{^m\jmath_{n,1}}_{!*}\ldots{^m\jmath_{n,n-b-1}}_{!*}{^m\jmath_{n,n-b}}_?\ldots {^m\jmath_{n,n-a-2}}_!{^m\jmath_{n,n-a-1}}_*{\imath_{n,n-a}}_*\kk_{\mathbb{C}^{n-a}[n-a]}
\end{equation}
where $?=\begin{cases}
    ! \ \hbox{if}\ a \equiv_2 b,\\ * \ \hbox{if}\ a \not\equiv_2 b.
\end{cases}$.

\begin{remark}
We note that when $b=a+1\neq n$ we recover the standard and costandard objects of the highest weight category $\perv{m}{X}$; in fact, we have 
\begin{equation*}
    \begin{split}
         \mathcal{M}(a,a+1,1) & \cong{{^m\jmath_{n,0}}_{!*}{^m\jmath_{n,1}}_{!*}\ldots{^m\jmath_{n,n-a-2}}_{!*}{\jmath_{n-a,n-a-1}}_! {\imath_{n,n-a}}_*\kk_{\C^{n-a}}[n-a] }\\ &\cong{\imath_{n-a}}_!\kk_{\C^{n-a}}[n-a]\cong\stan{}{a}
    \end{split}
\end{equation*}
and similarly
\[
\mathcal{M}(a,a+1,-1)={\imath_{n-a}}_*\kk_{\C^{n-a}}[n-a]\cong\costan{}{a}
\]
for $0\leq a< n-1$, while when $b=a=n$ we have $\stan{}{n}=\costan{}{n}=S_n$ since $\imath_{n,0}=\imath_0\colon S_0\to X$. 
\end{remark}

Note that it makes sense to say that a perverse sheaf $\mathcal{F}\in\perv{m}{X}$ is a brick, meaning that $\mathrm{End}_{\constr{c}{X}}(\mathcal{F})\cong \kk$. Moreover, observe that under the equivalence between $\perv{m}{X}$ and $\mod{B(n)}$, a perverse sheaf is a brick if and only if the corresponding $B(n)$-module is a brick.

\begin{lemma}
    For any $0\leq a < b \leq n$ the objects $\mathcal{M}(a,b,1)$ and $\mathcal{M}(a,b,-1)$ are  bricks.
\end{lemma}
\begin{proof}
    By \cite[Lemma 3.3.3]{Achar} the intermediate extension functor is fully faithful. By \cite[Proposition 1.4.17(iii)]{BBD}, the functors $^m{\jmath_{n,n-k}}_!$ and $^m{\jmath_{n,n-k}}_!$ are fully faithful. Finally, the functor ${\imath_{n,n-a}}_*$ is exact. Hence, $\End{\mathcal{M}(a,b,\pm1)}\cong\End{\kk_{\C^{n-a}}[n-a]}\cong\kk$. 
\end{proof}

\begin{remark}\label{rem:iterated_exts} By the definition of the intermediate extension functor, one can note that the thin objects in (\ref{thin_perverse1}) and (\ref{thin_perverse2}) are given by the unique non-zero extensions
    \[
    [\mathcal{M}(a,b,1)]\in\Ext{}{1}{\ic{}{a}}{\mathcal{M}(a+1,b, -1)}\quad\hbox{and}\quad[\mathcal{M}(a,b,-1)]\in\Ext{}{1}{\mathcal{M}(a+1,b,1)}{\ic{}{a}},
    \]
using the base case 
\[
    [\stan{}{a}]\in\Ext{}{1}{\ic{}{a}}{\ic{}{a-1}} \quad \hbox{and} \quad [\costan{}{a}]\in\Ext{}{1}{\ic{}{a-1}}{\ic{}{a}}.
    \] 
for $0\leq a <n$.
\end{remark}

\begin{proposition}\label{prop : perverse walls}
    A full list of the walls in the wall and chamber structure of $\perv{m}{\PP^n}$ is given by simple perverse sheaves and opportune iterated extensions of simple and standard/costandard objects given by the formulae (\ref{thin_perverse1}) and (\ref{thin_perverse2}).
\end{proposition}
\begin{proof}
The $n+1$ simple perverse sheaves $\ic{}{r}$ are evidently thin. These objects $\mathcal{M}(a,b,\eta)$ for $0\leq a<b\leq n$ and $\eta\in\{\pm1\}$ are $n(n+1)$ by the proof of Corollary \ref{cor : number of *classes of strings} and Lemma \ref{ lemma : parameterisation of strings}. Since they are pairwise non-isomorphic, they are all the walls in the wall and chamber structure of $\perv{m}{X}$ by Corollary \ref{cor:number_of_walls}. By Remark \ref{rem:iterated_exts} every object in this list can be built iteratively starting from simple, standard and costandard objects.
\end{proof}
    
The following remark stemmed from discussions with Rosanna Laking who also made us aware of the paper \cite{MR4545219}.

\begin{remark}
    For the reader having some familiarity with stability conditions or $\tau$-tilting theory, we point out that by the results of \cite{MR4545219} and Proposition \ref{prop : perverse walls} the simple objects in any heart obtained from $\perv{m}{\PP^n}$ by a finite sequence of tilts are (shifts of) bricks.   
\end{remark}


\begin{thebibliography}{10}

\bibitem{Achar}
        P. N. Achar.
        \newblock Perverse sheaves and representation theory.
        \newblock Mathematical Surveys and Monographs 258.
        \newblock American Mathematical Society, Providence, RI, 2021.

\bibitem{asai}
S. Asai. 
\newblock The wall-chamber structures of the real Grothendieck groups. 
\newblock Adv. Math.,
381:107615, 2021

\bibitem{MR2197389}  
I. Assem, D. Simson, A. Skowroński.
\newblock Elements of the representation theory of associative algebras. Vol. 1. Techniques of representation theory. London Mathematical Society Student Texts, 65. Cambridge University Press, Cambridge, 2006. 

\bibitem{BBD}
A. Be\u{\i}linson, J. Bernstein and P. Deligne.
\newblock Faisceaux pervers.
\newblock In {\em Analysis and topology on singular spaces, {I} ({L}uminy,1981)}, volume 100 of {\em Ast\'erisque}, pages 5--171. Soc. Math. France,
 Paris, 1982.

\bibitem{MR2119139}
A. Be\u{\i}linson, R. Bezrukavnikov, and I. Mirkovi\'{c}.
\newblock Tilting exercises.
\newblock {\em Mosc. Math. J.}, 4(3):547--557, 782, 2004.

\bibitem{MR1322847}
A. Be\u{\i}linson, V. Ginzburg, W. Soergel.
\newblock Koszul duality patterns in representation theory.
\newblock J. Amer. Math. Soc.9(1996), no.2, 473--527.

\bibitem{bousseau}
P. Bousseau.
\newblock Scattering diagrams, stability conditions, and coherent sheaves on $\mathbb{P}^2$,
\newblock 
Journal of Algebraic Geometry 31 (2022), no. 4, 593--686.

\bibitem{bridgelandWC}
T. Bridgeland. 
\newblock Scattering diagrams, hall algebras and stability conditions. \newblock Algebr. Geom. 4 (2017), 523--561

\bibitem{bridgeland07}
T. Bridgeland
\newblock Stability Conditions on Triangulated Categories.
\newblock Annals of Mathematics 166, no. 2 (2007), 317--345. 

\bibitem{BST19}
T. Br\"ustle, D. Smith,  H. Treffinger. 
\newblock Wall and chamber structure for finite-
dimensional algebras. 
\newblock Adv. Math., 354:106746, 31, 2019.

\bibitem{Butler-Ringel}
M.C.R. Butler, C.M. Ringel.
\newblock Auslander--Reiten sequences with few middle terms and applications to string algebras.
\newblock Comm. Algebra 15 (1987), no. 1--2, 145–-179. 

\bibitem{MR4349392}
A. Cipriani, J. Woolf.
\newblock When are there enough projective perverse sheaves?
\newblock Glasg. Math. J. 64 (2022), no. 1, 185--196.

\bibitem{DCM}
M.A. de Cataldo, L. Migliorini.
\newblock The decomposition theorem, perverse sheaves and the topology of algebraic maps.
\newblock Bull. Amer. Math. Soc. (N.S.) 46 (2009), no. 4, 535--633.

\bibitem{douglas}
M.R. Douglas.
\newblock Dirichlet branes, homological mirror symmetry, and stability.
\newblock Proceedings of ICM 2002, vol. 3, 395--408.

\bibitem{gm2}
M. Goresky, R. MacPherson.
\newblock Intersection homology ii.
\newblock {\em Inventiones mathematicae}, 72(1):77--129, 1983.

\bibitem{WC_structure}
M. Kaipel, H. Treffinger.
\newblock Wall-and-chamber structures for finite-dimensional algebras and $\tau$-tilting theory.
\newblock \href{https://arxiv.org/pdf/2302.12699.pdf}{arXiv:2302.12699}

\bibitem{king}
A.D. King.
\newblock Moduli of representations of finite dimensional algebras. 
\newblock Quart. J. Math. Oxford Ser. (2) 45 (1994), no. 180, 515--530.

\bibitem{KW}
F. Kirwan, J. Woolf.
\newblock An introduction to intersection homology theory.
\newblock Chapman \& Hall/CRC, Boca Raton, FL, 2006, xiv+229 pp.

\bibitem{MR1862802}
M. Khovanov, P. Seidel.
\newblock Quivers, Floer cohomology, and braid group actions.
\newblock J. Amer. Math. Soc.15(2002), no.1, 203--271.

\bibitem{Pe}E. Persson Westin.
\newblock
Tilting modules and exceptional sequences for leaf quotients of type A zig-zag algebras.
\newblock
Beitr. Algebra Geom.61(2020), no.2, 189--207.


\bibitem{string_and_band}
S. Schroll, A. Tattar, H. Treffinger, Y. Valdivieso, N.J. Williams.
\newblock Stability spaces of string and band modules. 
\newblock 
Journal of Pure and Applied Algebra, 107503.

\bibitem{MR4545219}
F. Sentieri.
\newblock A brick version of a theorem of Auslander.
\newblock Nagoya Math. J.249(2023), 88--106.

\bibitem{SW}
A. Skowro\'nski, J. Waschb\"usch.
\newblock Representation-finite biserial algebras.
\newblock J. Reine Angew. Math. 345 (1983) 172--181.

\bibitem{stroppelTQFT}
C. Stroppel.
\newblock TQFT with corners and tilting functors in the Kac-Moody case.
\newblock \href{https://arxiv.org/abs/math/0605103}{arXiv:0605103}

\bibitem{MR2739061}
J. Woolf.
\newblock Stability conditions, torsion theories and tilting. 
\newblock J. Lond. Math. Soc. (2) 82 (2010), no. 3, 663--682. 

\bibitem{WW}
B. Wald, J. Waschb\"{u}sch.
\newblock Tame biserial algebras.
\newblock J. Algebra 95 (1985), no. 2, 480--500. 

\end{thebibliography}
\end{document}